\documentclass[12pt]{amsart}

\numberwithin{equation}{section}

\usepackage{amsmath, amssymb, amsfonts, amscd, mathtools}
\usepackage{bm}            % Bold math symbols
\usepackage{mathrsfs}      % Script fonts
\usepackage{calligra}      % Calligraphic font for custom operators

\usepackage{graphicx}
\usepackage{tikz-cd}
\usepackage[all]{xy}
\usepackage{subcaption}
\usepackage{adjustbox}

\usepackage{enumerate}    % Better list control
\usepackage{listings}      % Code listings
\usepackage{subfiles}      % Modular documents
\usepackage[toc,page]{appendix}
\usepackage{verbatim}
\usepackage{comment}

\usepackage{hyperref}
\hypersetup{
    colorlinks=true,
    citecolor=red,
    linkcolor=blue,
    filecolor=magenta,
    urlcolor=red,
}
\usepackage[capitalize,nameinlink,noabbrev]{cleveref}

\newcommand{\C}{\mathbb{C}}

\newcommand{\Q}{\mathbb{Q}}

\newcommand{\Z}{\mathbb{Z}}

\newcommand{\cF}{\mathcal{F}}

\newcommand{\cO}{\mathcal{O}}

\newcommand{\D}{\Delta}

\DeclareMathOperator{\Supp}{Supp}
\DeclareMathOperator{\codim}{codim}
\DeclareMathOperator{\mult}{mult}

\DeclareMathOperator{\HHom}{\mathscr{H}\text{\kern -3pt {\calligra\large om}}\,}

\newtheorem{theorem}{Theorem}[section]
\newtheorem{lemma}[theorem]{Lemma}
\newtheorem{proposition}[theorem]{Proposition}
\newtheorem{corollary}[theorem]{Corollary}

\newtheorem{set-up}[theorem]{Set-up}
\newtheorem{definition}[theorem]{Definition}
\newtheorem{example}[theorem]{Example}
\newtheorem{remark}[theorem]{Remark}
\newtheorem{question}[theorem]{Question}

\title{Rank one foliations on toroidal varieties}

\author{Calum Spicer}
\address{Department of Mathematics, King's College London, Strand,
London WC2R 2LS, UK}
\email{calum.spicer@kcl.ac.uk}

\author{Luca Tasin}
\address{Dipartimento di Matematica ``F. Enriques'', Universit\`a degli Studi di Milano, Via Saldini 50, 20133 Milano (MI), Italy}
\email{luca.tasin@unimi.it}
\begin{document}

\begin{abstract}
Consider a log canonical pair $(X,B)$ such that there is a Cartier divisor $D$ for which $T_X(-\log B) \otimes \mathcal O(D)$ is locally free and globally generated. Let $\mathcal F$ be a log canonical foliation of rank 1 on $X$. We prove that there exists a divisor $\Gamma$ such that $(X, \Gamma)$ is log canonical and $K_X + \Gamma \sim K_{\mathcal F} + D$.

We then apply this result to prove several statements on the birational geometry of rank 1 log canonical foliations on log homogeneous varieties. 
\end{abstract}

\maketitle
\tableofcontents
\section{Introduction}

A central theme in the birational geometry of foliations is that the behaviour of a foliation $\cF$ is governed by its canonical divisor $K_\cF$. On the other hand, the main tools of birational geometry are formulated for log canonical pairs $(X,\D)$, and, especially in questions of boundedness and volume, are not directly applicable in the foliated setting.
The purpose of this paper is to bridge this gap for rank 1 foliations on toroidal varieties.

\begin{theorem}\label{thm:main_intro}
Let $X$ be a projective variety and $B$ a reduced divisor on $X$. Assume that:
    \begin{enumerate}
    \item $(X, B)$ is log canonical;
        \item there is a Cartier divisor $D$ on $X$ such that $T_X(-\log B) \otimes \mathcal O(D)$ is locally free and globally generated;
        \item there is a rank 1 foliation $\cF$ and an effective $\Z$-divisor $\Delta$ such that $(\cF,\Delta)$ has log canonical singularities. 
    \end{enumerate} 
    
Then there exists a reduced divisor $\Gamma$  such that $(X, \Delta +  \Gamma)$ has log canonical singularities and 
\[
K_X+\Delta + \Gamma \sim K_{\mathcal F}+\Delta + D.
\] 
\end{theorem}

Note that conditions (1) and (2) imply that $(X, B)$ is toroidal, cf. \cite[Theorem 6.8]{Bergner}. 

One important situation where Theorem \ref{thm:main_intro} applies is the case where $D = 0$. In this case, the variety $X$ is called log homogeneous.
Log homogeneous varieties form a particularly interesting class of algebraic varieties, which include spherical varieties, homogeneous spaces, and toric varieties.
In the smooth set-up, such varieties have been intensively studied, see \cite{Brion} for a general introduction.
We also remark that if $X$ is smooth (or toroidal) and $B=0$, one may take $D$ sufficiently ample to satisfy the assumptions of the theorem. However, such a choice is unlikely to yield interesting consequences. Instead, $D$ should be chosen in a way that is controlled by the geometry of $\mathcal{F}$ or $X$.

We derive from our main theorem several corollaries on the birational geometry of foliations on log homogeneous varieties; in particular, we verify some outstanding conjectures in the field for these classes of foliations.  We describe some of these applications now.

\subsection{Volume and boundedness} 

Given two positive integers $n, r$  and a DCC set $I \subset \mathbb Q$, consider the sets

\begin{align*}
\mathcal F_{n,r, I} = \{(X, \D, \mathcal F) : ~ & \bullet X \text{ is a normal variety and }\dim X=n, ~ \\
    &\bullet \mathcal (\mathcal F, \D) \text{ is a log canonical rank $r$ foliated pair}\\ ~ 
    %&\bullet \D \in I , \\
     &\bullet \D \in I , \mbox{ and } K_{\mathcal F}+\D \text{ is big and nef.}\}\end{align*}

and

\begin{align*}
\mathcal T_{n,r} = \{(X, \D, \mathcal F) : ~ & \bullet (X, \Delta, \mathcal F) \in \mathcal F_{n, r, I=\{1\}}, ~ \\
    &\bullet \text{there exists a reduced divisor} ~B ~\text{such that}~ \\ & \indent  (X, B) \text{~ is log homogeneous.}\}\end{align*}

We emphasise that for $(X, \D, \mathcal F) \in \mathcal T_{n, r}$ we are {\it not} assuming that $\mathcal F$ is equivariant with respect to any group action on $X$.

Given the importance of the volume in the study of birational geometry \cite{HMX10, HMX12}, the following very natural question was asked by J. V. Pereira, see also \cite[\S 3]{MR4218629} or \cite[Question 0.4]{HL19}.

\begin{question}
\label{main_question}
    Fix $n, r \in \mathbb Z_{>0}$ and a DCC set $I \subset [0,1]$.  Does the set \[\mathcal V(n, r):=\{(K_{\mathcal F}+\D)^n : (X, \D, \mathcal F) \in \mathcal F_{n,r,I}\}\]
    satisfy the descending chain condition (DCC)?
\end{question}

Very little is known about Question \ref{main_question}, even in the case $n=2$.  For instance, it is not even known if $0< v_0 := \inf \mathcal V(2, 1)$.
Indeed, this is even unknown for surface foliations defined by fibrations.
We refer to \cite{HL19, PS16, SS23, Chen, Lu-Tan} for some partial results on understanding the structure of $\mathcal V(2, 1)$. 
In any dimension, a positive answer is known if the foliation is algebraically integrable with bounded leaves \cite{HJLL} or birationally bounded \cite{Fan25}.
The existence of $v_0$ is proven in \cite{CPT} in any dimension if the foliation is given by a fibration with reduced fibres by combining positivity results for Hodge bundles and the slope inequalities from \cite{CTV}.  

We emphasize that this problem for foliations is more subtle than the analogous one for varieties, because birational boundedness fails in general for foliations; see \cite{Luu,Passantino}.

\begin{corollary}[DCC of volume]
\label{cor:DCC}
   Fix a positive integer $n$. Then 
   \begin{enumerate}
       \item   the set  \begin{align*}
     \{ (K_{\mathcal F}+ \D )^n: (X, \D, \mathcal F) \in \mathcal T_{n,1}
    \}
    \end{align*}
        satisfies the descending chain condition;
\item there exists a positive integer $m$ such that the map $\phi_{m(K_{\mathcal F}+\D)}$ is birational. 
   \end{enumerate}
        
\end{corollary}

To the best of our knowledge, this is the first evidence for an affirmative answer to Question \ref{main_question}
when the foliation is {\it not} assumed to be algebraically integrable.

We also remark that the log canonicity condition on the foliation plays a critical role.  Indeed, without the log canonicity condition the volume does not fall into a DCC set, as Example \ref{ex_toric_not_lc} shows.

\begin{corollary}[Boundedness of canonical models]
\label{cor:boundedness}
Fix a positive integer $n$, and a positive real number $v$. Then the set of triples $(X,\D,\cF)$ such that
$$
(X,\D,\cF) \in \mathcal T_{n,1}, \quad K_\cF+\D \text{ is ample, } \text{ and } (K_\cF + \D)^n=v
$$
form a bounded family. 
\end{corollary}

\subsection{Base point free theorem and MMP }

The cone theorem for rank 1 foliations is proven in full generality in \cite{CS23b}.  
In \cite{CS25}, the same authors proved the base point free theorem for log canonical foliated pairs of rank 1 $(\cF, \D)$ such that $\D=A+B$ where $A$ is ample and $B$ is effective.

As an immediate consequence of Theorem \ref{thm:main_intro}  we get the base point free theorem for  log canonical foliated pairs of rank 1 $(\mathcal F, \Delta)$ on log homogeneous varieties. In addition, we can prove that the MMP works in the following sense. Recall that a variety $X$ is log homogeneous if and only if there exists a semi-abelian group $G$ which acts on $X$ with $X \setminus B$ as an open orbit (see \cite[Corollary 6.3]{Bergner}). 

\begin{corollary}[MMP for log homogeneous foliations]\label{cor:MMP}
Let $X$ be a log homogeneous variety with associated group $G$. Let $\cF$ be a rank one foliation and let $\D$ be a reduced divisor such that $\cF$ and $\D$ are $G$-invariant and $(\cF, \D)$ is log canonical. Then there exists a $K_{\mathcal F}+\D$ MMP  such that each step $(X_i, \D_i, \cF_i)$ is log homogeneous. 
\end{corollary}

Contrary to the rest of the paper, in Corollary \ref{cor:MMP} we assume that the pair $(\cF, \D)$ is equivariant; this assumption is necessary, since otherwise the MMP may leave the category of log homogeneous varieties. Moreover, if $X$ is toric, the above corollary applies to toric foliations. In such set-up, more general and precise results are known, see \cite{CC25,MR4986833, Spicer20, MR4600107}. We also remark that not every log homogeneous variety is a Mori dream space; therefore, Corollary \ref{cor:MMP} does not follow from the general theory of Mori dream spaces.

\subsection{Adjoint foliated structures}
We refer to \cite{cascini2024, C7} for the definition of adjoint foliated structures, but we remark that they have proven to be important objects of study in the birational geometry of foliations.  Most of the existing literature on these objects focuses on algebraically integrable adjoint foliated structures.  We are able to deduce here many important results for adjoint foliated structures on log homogeneous varieties which were previously only known in the algebraically integrable setting.

\begin{corollary}[Boundedness of minimal models]\label{cor:adjoint_minimal} 
Fix $n \in \mathbb Z_{>0}$, $v \in \mathbb R_{>0}$ and a DCC set $I \subset [0,1)$. Then the pairs $(X,\cF)$ such that
\begin{enumerate}
    \item $(X,B)$ is log homogenous of dimension $n$ and $K_X$ is $\Q$-Cartier;
    \item $\cF$ is a log canonical foliation on $X$ of rank 1;
    \item $tK_\cF + (1-t)K_X$ is nef and big with $(tK_\cF + (1-t)K_X)^n=v$ for some $t \in I$,
\end{enumerate}
form a bounded family.     
\end{corollary}

\begin{corollary}[Boundedness of Fanos, cf. {\cite[Theorem B]{C7}}]\label{cor:adjoint_Fano}
Fix $n \in \mathbb Z_{>0}$ and $\varepsilon \in \mathbb R_{>0}$. Then the pairs $(X,\cF)$ such that
\begin{enumerate}
    \item $(X,B)$ is log homogenous of dimension $n$ and $X$ is $\varepsilon$-lc;
    \item $\cF$ is a log canonical foliation on $X$ of rank 1; and
    \item $-(tK_\cF + (1-t)K_X)$ is ample for some $ 0 < t \le 1-\varepsilon$, 
\end{enumerate}
form a bounded family.         
\end{corollary}

\begin{corollary}[Fano type, cf. {\cite[Theorem C]{C7}}]\label{cor:Fano}
Let $\cF$ be a rank 1 log canonical foliation on a log homogenous variety $X$ with $K_X$ $\Q$-Cartier. Assume that $-(tK_\cF + (1-t)K_X)$ is ample for some $t \in [0,1]$. Then $X$ is of Fano type, i.e. there exists $\D$ such that $(X,\D)$ is klt and $-(K_X+\D)$ is ample. 
\end{corollary}

\subsection{Sketch of proof}
The basic idea of the proof is to consider the tangency locus between $\mathcal F$ and a rank $n-1$ distribution on $X$ which is generated by a general choice of $n-1$ global vector fields.  
This tangency locus gives a divisor $\Gamma$ such that $\Gamma \sim c_1(N_{\mathcal F})+D$ and hence $K_X+\Gamma \sim K_{\mathcal F}+D$. 

The main issue is showing that the pair $(X, \Gamma)$ has mild singularities.  In the case where $\mathcal F$
has rank one there is a criterion due to McQuillan-Panazzolo \cite{MP13} which shows that a foliation has log canonical singularities provided that the foliation induces a non-nilpotent endomorphism on the Zariski tangent space at each of the singular points of the foliation. 
To illustrate this control consider the example 
where $X = \mathbb C^2$, $B = \{xy = 0\}$ and $\mathcal F$ is generated by a vector field $a\partial_x+b\partial_y$ which is singular at the origin.
The tangency locus between a section of $T_X(-\log B)$
of the form $\lambda x\partial_x+\mu y\partial_y$
and $\mathcal F$ is given by 
$\{\mu xb- \lambda ya = 0\}$.  Since $\partial$ defines a log canonical foliation,  the matrix 
\[
\begin{pmatrix}
a_x(0) & a_y(0) \\
b_x(0) & b_y(0)
\end{pmatrix}\]
is non-nilpotent, hence for a general choice of $\mu$ and $\lambda$, 
\[\mu xb- \lambda ya = \mu b_x(0)x^2- \lambda a_y(0)y^2 +(\mu b_y(0) - \lambda a_x(0))xy+\text{h.o.t.}\] defines a curve with at worst a normal crossing singularity at $0$.
In higher dimensions this calculation is more involved and is covered in Section \ref{s_lin_alg}.

This also explains why our method does not readily generalise to higher rank foliations.  Indeed, there is no obvious analogue of McQuillan-Panazzolo's criterion for higher rank foliations.
In principle, however, our main theorem should hold for higher rank foliations so we ask the following.

\begin{question}
  Does Theorem \ref{thm:main_intro}  hold for higher rank foliations?
\end{question}

\medskip

\textbf{Acknowledgment.}
We thank P.\ Cascini, J. Liu and R. Svaldi for helpful discussions and comments. 

The first author was partially supported by EPSRC. The second author was supported by PRIN2020 research grant ”2020KKWT53” and is a member of the GNSAGA group of INdAM. He also thanks King's College London for their hospitality during his research visit.

\section{Singularities of rank 1 foliations}

\textbf{Notations.}
We work over the field of complex numbers $\mathbb C$. 
We refer to \cite{KM98} for the classical definitions of singularities that appear in the minimal model program.

Given a normal variety $X$, we denote by $\Omega^1_X$ its sheaf of K\"ahler differentials and
by $T_X:=(\Omega^1_X)^*$ its tangent sheaf. 
A {\bf foliation of rank one} on a normal variety $X$ is a rank one coherent subsheaf $T_{\mathcal F}\subset T_X$ such that 
$T_{\mathcal F}$ is saturated in $T_X$.
The {\bf canonical divisor} of $\mathcal F$ is a divisor $K_{\mathcal F}$ such that $\mathcal O_X(-K_{\mathcal F})\simeq  T_{\mathcal F}$.  
A {\bf rank one foliated  pair} $(\mathcal F, \Delta)$ is a pair of a foliation $\mathcal F$ of rank one and a $\mathbb Q$-divisor $\Delta\ge 0$
such that $K_{\mathcal F}+\Delta$ is $\mathbb Q$-Cartier. 
We refer to \cite[Section 2.2 and Section 2.3]{CS20} for basic notions regarding foliations, including the definition of classes of singularities coming from the MMP.

Let $x \in X$ be a germ of a smooth variety and $\partial$ be a vector field on $X$ vanishing at $x$, which defines a foliation $\mathcal F$. Denote with $\mathfrak m$ the maximal ideal at $x$ and note that $\mathfrak m^n$ is $\partial$-invariant for any $n \ge 1$. In particular,  we get an induced linear map
$$
\partial_0 : \mathfrak m/ \mathfrak m^2 \to  \mathfrak m/ \mathfrak m^2,
$$
that we call the \textbf{linear part} of the foliation $\mathcal F$. We refer to \cite[\S 2.8]{CS20} for more details.

\medskip

The following proposition collects some fundamental facts on singularities of rank 1 foliations. 

\begin{proposition}\label{prop_singularities}
Let $x \in X$ be a germ of a normal variety and let $\cF$ be a germ of a rank 1 foliation on $X$ such that $K_{\mathcal F}$ is $\mathbb Q$-Cartier. 
\begin{enumerate}
    \item Let $\sigma\colon X'\to X$
    be the index one cover associated to $K_{\mathcal F}$ and let $\mathcal F'= \sigma^{-1}\mathcal F$. Then, $\mathcal F$ is terminal (resp. log canonical) if and only if $\mathcal F'$ is terminal (resp. log canonical).
    \item If $x \in X$ is smooth, then $\cF$ is terminal at $x$ if and only if $x \notin {\rm Sing}~ \mathcal F$.
    \item If $K_{\mathcal F}$ is Cartier (equivalently $\mathcal F$ is defined by a vector field $\partial$), if $x \in {\rm Sing} ~\mathcal F$
    and $\partial_0$ is the linear
    part of $\partial$ at $x$, then $\mathcal F$ is log canonical at $x$ if and only if $\partial_0$ is non-nilpotent.
\end{enumerate}

\end{proposition}
\begin{proof}
(1) See, e.g., \cite[Lemma 2.20]{SS23}.

(2) This is \cite[Fact III.i.1]{MP13}.

 (3) This is  \cite[Fact I.ii.4]{MP13}.
\end{proof}

\begin{lemma}\label{lem_support}
Let $x \in X$ be a germ of a smooth variety, let $\mathcal F$ be a rank one foliation on $X$ and let $\D \ge 0$ be a $\mathbb R$-divisor.

Assume that $(\mathcal F, \Delta)$ is log canonical and $x \in \mathrm{Sing}~ \mathcal F$.  Then $x \notin \Supp(\D)$.
\end{lemma}
\begin{proof}
Since $x \in X$ is smooth and $x \in \mathrm{Sing}(\mathcal F)$, Proposition \ref{prop_singularities} implies that the foliated discrepancy of $x \in (X,\cF)$ is less or equal to $-\varepsilon(E)$.

Consider the blowup $\pi: X' \to X$ of $x$ and let $\cF'=\pi^{-1} \cF$ and $\D'$ be the strict transform of $\D$. Then
$$
K_\cF + \D' = \pi^*(K_\cF + \D) +a E
$$
where $E$ is the exceptional divisor. 
Since $x \in (X,\D,\cF)$ is log canonical we must have  $x \notin \Supp(\D)$, otherwise $a < -\varepsilon(E)$.
\end{proof}

\subsection{Deformation to the normal cone}

\begin{lemma}
\label{lem_def_normal_cone}
    Consider the following set up.
    \begin{itemize}
        \item Let $(\mathbb A^n, \Delta = \sum_{i = 1}^N d_iD^i)$ be a pair where $D^i$ is an irreducible Weil divisor and for all $i$, $0 \in D^i$.

        \item Let $X = \mathbb A^n \times \mathbb A^1$ with coordinates $(x_1, \dots, x_n, t)$.
%        \item Let $\Theta \ge 0$ be an $\mathbb R$-divisor on $\mathbb A^n$ and let $\overline{\Theta} = \Theta \times \mathbb A^1$.
        \item Let $\mathbb G_m$ act on $X$ by 
        \[\lambda \cdot (x_1, \dots, x_n, t) = (\lambda^{w_1} x_1, \dots, \lambda^{w_n} x_n, \lambda^{-a}t)\]
        where $w_1, \dots, w_n \in \mathbb Z_{\ge 0}$ and $a \in \mathbb Z_{>0}$.
        \item Let $\overline{D^i}$ be the $\mathbb G_m$-orbit closure of 
        $D^i \times \{1\} \subset \mathbb A^n \times \{1\}$.
        \item Let $D^i_0$ be the fibre of $\overline{D^i}$ over $t = 0$.
    \end{itemize}

    Then the following hold.

    \begin{enumerate}
        \item  If $(\mathbb A^n, \sum d_iD^i_0)$ is log canonical then $(\mathbb A^n, \Delta)$
    is log canonical.  

    \item Let $D^i = \{f_i = 0\}$ and let $f^0_i$ be the homogeneous part of $f_i$ of lowest degree. Set $E_i := \{f_i^0 =0\}$.  If $(\mathbb A^n, \sum d_iE_i)$ is log canonical, then $(\mathbb A^n, \Delta)$
    is log canonical.  
    \end{enumerate}

\end{lemma}
\begin{proof}
Let us first prove (1).
    If $D^i_t$ denotes the fibre of $\overline{D^i}$ over $t \in \mathbb A^1$
    then observe by construction that 
    \begin{align}\label{iso}(\mathbb A^n, \sum d_iD^i_t) \cong (\mathbb A^n, \sum d_iD^i)\end{align}for any $t \ne0$.

    If $(\mathbb A^n, \sum d_iD^i_0)$ is log canonical, then inversion of adjunction of singularities, see \cite{Kawakita05}, implies that $(X, \sum d_i\overline{D^i})$ is log canonical in a neighbourhood of $\mathbb A_n \times \{0\}$.  Adjunction on singularities
    then implies that for $t$ near $0$ we have $(\mathbb A^n, \sum d_iD^i_t)$ is log canonical 
    and we can conclude by using the isomorphism (\ref{iso}).

    \medskip

    We now observe that (2) is an immediate consequence of (1).  Take $w_1 = \dots = w_n = a = 1$  and observe that $\overline{D^i} = \{t^{-\deg f_i^0} f_i(tx_j) = 0\}$ and so we see that $D^i_0 = E_i$.  We may now apply (1) to deduce (2).
\end{proof}

\begin{remark}
    We remark that Lemma \ref{lem_def_normal_cone} holds if
    $\mathbb A^n$ is replaced with a polydisk $\mathbb D(0, r) \subset \mathbb C^n$ and all divisors are assumed 
    to be analytic divisors.  Indeed, the same proofs work, the only change being needed is to use \cite[Theorem 1.1]{MR4702587} in place of \cite{Kawakita05}. 
\end{remark}

\section{Some linear algebra}
\label{s_lin_alg}

Let $A =(a_{ij}) \in M_n(\mathbb K)$ be an $n \times n$ matrix with coefficients in a field $\mathbb K$.
Given a subset $I \subset \{1,\ldots,n\}$ of cardinality $k$, the principal submatrix $A(I)$ of $A$ of order $k$ is obtained by keeping only the rows and columns of $A$ whose indices lie in $I$.
To $A$ we associate a directed graph $G(A)$ with vertex set $\{1,\dots,n\}$ and an edge $i\to j$ whenever $a_{ij}\neq 0$.

\begin{lemma}\label{lem_special}
Let $A\in M_n(\mathbb{K})$ be a matrix such that every proper principal submatrix of $A$ is nilpotent. Then the following are equivalent:

\begin{enumerate}[(i)]
    \item $A$ is not nilpotent;
    \item  $G(A)$ admits a directed cycle of length $n$ and no directed cycle of length less than $n$;
    \item  up to a permutation of the index set $\{1,\dots,n\}$, the matrix $A$ has the form
\[
A=\begin{pmatrix}
0 & a_{12} & 0 & \cdots & 0 \\
0 & 0 & a_{23} & \cdots & 0 \\
\vdots & & \ddots & \ddots & \vdots \\
0 & 0 & \cdots & 0 & a_{n-1,n} \\
a_{n1} & 0 & \cdots & 0 & 0
\end{pmatrix},
\]
with $a_{i,i+1}\neq 0$ for $1\le i\le n-1$ and $a_{n1}\neq 0$.
\end{enumerate}
\end{lemma}

\begin{proof}
We argue by induction on $n$ and the case $n=1$ is trivial.
Assume $n\ge2$ and that the statement holds for all matrices of order less than $n$ satisfying the same hypothesis.

$(i) \Rightarrow (ii)$  If $G(A)$ has no directed cycle, then it is acyclic and hence admits an initial vertex.
After permuting indices we may assume that the initial vertex $=1$ and so $A$ is strictly upper triangular, hence nilpotent, contradicting $(i)$.
Thus $G(A)$ contains at least one directed cycle.

Let $k$ be the minimal length of a directed cycle in $G(A)$. The vertices of such a cycle determine a principal submatrix $B$ of order $k$. By minimality of $k$, the graph $G(B)$ has no cycle of length less than $k$, and by hypothesis every proper principal submatrix of $B$ is nilpotent.
By the inductive hypothesis applied to $B$, it follows that $B$ is not nilpotent. Thus, $B=A$ and $k=n$.

\medskip

$(ii) \Rightarrow (iii)$
Since $G(A)$ contains a cycle of length $n$, up to reordering we may assume that this cycle is 
\begin{align}
\label{cycle}
    1 \to 2 \to \dots \dots \to  n \to 1.
\end{align}
We next observe that $A$
has no other non-zero entries.  Indeed, suppose for sake of contradiction that $A$ had some other non-zero entry. This gives an edge $i \to j$ where either $j \neq i+1$ or $i = n$ and $j \neq 1$.
In either case, the existence of this edge implies the existence of a directed cycle in $G(A)$ of length $<n$, contrary to hypothesis.
Thus the graph $G(A)$ is precisely the cycle \eqref{cycle} and so $A$ is in the claimed form.

\medskip

$(iii) \Rightarrow (i)$  This is clear, since the characteristic polynomial of $A$ is 
$$
\det(tI-A)=t^n - a_{12}a_{23}\cdots a_{n-1,n} a_{n1}.
$$
\end{proof}

\begin{proposition}\label{prop_det}
Let $n$ be a positive integer and $r$ an integer such that $0 \le r \le n$. Let $A = (a_{ij})_{1 \le i,j \le n} \in M_n(\mathbb C)$ be a non-nilpotent matrix. Set $A_i := \sum _{j=1}^n a_{ij}x_j$. Then there exist  $\lambda_1, \ldots, \lambda_n$ not all zero such that
$$
f(x_1,\dots,x_n) := \sum_{i=1}^n \lambda_i A_i x_1\cdots \hat x_{i} \cdots x_r= (x_1 \cdots x_r)(\sum_{i=1}^r \lambda_i \frac{A_i}{x_i}+ \sum_{i=r+1}^n \lambda_i A_i)
$$
defines a hypersurface in $\mathbb C^n$ which is log canonical at the origin. 
\end{proposition}

\begin{proof}
If $r=0$, then 
$$
f(x)=\sum \lambda_i A_i
$$
is linear and non-zero for a general choice of $\lambda_i$, so $\{f=0\}$ is actually smooth at the origin. Hence we can assume $r \ge 1$ from now on. 

We argue by induction on $n \ge 1$. If $n=1$, then
$$
f(x)= \lambda_1 A_1= \lambda_1 a_{11}x_1,
$$
and we can simply take $\lambda_1 \ne 0$. Hence we can assume $n\ge2$ and that the statement holds for any $k <n$.

\medskip

\textbf{Case 1: every proper principal submatrix of $A$ is nilpotent.} Then $A$ is of the form given by Lemma \ref{lem_special}. Note that the order might be different.

If $1 \le r<n$, then the principal top-left submatrix of order $r$ of $A$ must have a zero row, otherwise $G(A)$ would have a cycle of length  $r < n$, which contradicts Lemma \ref{lem_special}. Hence there is $i\le r$ and $\ell > r$ such that $A_i=a_{i,\ell}x_{\ell}$ and taking $x_i=0$, we get
$$
f(x_1, \ldots, x_i=0, \ldots, x_n)= \lambda_i a_{i,\ell} x_\ell x_1 \cdots \hat x_i \cdots x_r,
$$
which gives a log canonical singularity. By inversion of adjunction, $f(x)=0$ is also log canonical for $\lambda_i \ne 0$.

We now deal with the case $r=n$, in which a similar trick does not work. 
Then (up to reordering $\{1,\ldots, n\}$)
\[
f(x)
=
\lambda_1 a_{12}x_2^2x_3 x_4 \cdots x_n
+
\lambda_2 a_{23} x_1x_3^2x_4\cdots x_n
+
\cdots
+
\lambda_n a_{n1} x_1^2x_2x_3 \cdots x_{n-1},
\]

Let $\mathfrak a$ be the monomial ideal generated by the monomials appearing in $f$, i.e.
\[
\mathfrak a
=
(x_1 \cdots x_{i-2} x_i^2 x_{i+1}\cdots x_n)_{i=1, \ldots,n}.
\]

The Newton polyhedron $P(\mathfrak a)$ is the convex hull of the exponent vectors
\[
(2,0,1,\ldots,1), (1,2,0,1,\ldots,1) \ldots, (0,1,\ldots,1,2),
\]
and so
\[
(1,\ldots,1) = \frac{1}{n}(2,0,1,\ldots,1) +\ldots+\frac{1}{n} (0,1,\ldots,1,2)\in P(\mathfrak a).
\]
By Howald's theorem on multiplier ideals of monomial ideals \cite[Main Theorem]{MR1828466},
\[
\mathrm{lct}_0(\mathfrak a)=1.
\]

If we take $f$ to be a general linear combination of the monomial generators of $\mathfrak a$,
the log canonical threshold of $f$ at the origin equals $\mathrm{lct}_0(\mathfrak a)$, see \cite[Example 1.10]{Mustata}.
Therefore
\[
\mathrm{lct}_0(f)
=
\mathrm{lct}_0(\mathfrak a)
=
1.
\]

Hence the hypersurface $\{f=0\}$ is log canonical at the origin.

\medskip

\textbf{Case 2: $A$ has a principal submatrix $A'$ of order $1\le k<n$ that is not nilpotent.}

Let $I=\{i_1,\dots,i_k\}\subset \{1,\dots,n\}$ be the index set of this
principal submatrix, so that $A'=(a_{ij})_{i,j\in I}$.
Set $\lambda_i=0$ for $i\notin I$ and
\[
K=\{1,\dots,r\}\setminus I .
\]
From the definition of $f$ it follows that
\[
f=\Big(\prod_{j\in K} x_j\Big)\, g ,
\]
where
\[
g=\sum_{i\in I}\lambda_i A_i
\prod_{\substack{1\le j\le r\\ j\ne i,\, j\notin K}} x_j .
\]

Consider the coordinate subspace
\[
L=\{x_j=0 \mid j\notin I\}\subset \mathbb C^n .
\]
On $L$ the linear forms $A_i$ with $i\in I$ restrict to
\[
A_i|_L=\sum_{j\in I} a_{ij}x_j ,
\]
which are precisely the linear forms associated with the matrix $A'$.
Hence $g|_L$ is the polynomial obtained from $A'$ by the same
construction (with the variables $\{x_i\}_{i\in I}$).

Since $A'$ is not nilpotent, the inductive hypothesis provides
$\lambda_{i_1},\dots,\lambda_{i_k}$, not all zero, such that
\[
\{g|_L=0\}\subset L\simeq \mathbb C^k
\]
is log canonical at the origin.

Finally, since
\[
f=\Big(\prod_{j\in K} x_j\Big) g,
\]
repeated applications of inversion of adjunction along the coordinate
hyperplanes $\{x_j=0\}$ for $j\in K$ show that $\{f=0\}\subset\mathbb C^n$
is log canonical at the origin.
\end{proof}

\section{Pairs associated to foliations}

The basic idea of this section is to study the tangency locus between a foliation on $X$ with log canonical singularities and general vector fields on $X$.  The main result of this paper, 
Theorem \ref{thm:main_intro}, shows that under suitable hypotheses we can control the singularities of
 the tangency locus.

\begin{definition}
Let $X$ be a normal variety of dimension $n$, let $\mathcal F$ be a foliation of rank $r$ on $X$
and let $\mathcal G$ be a distribution of rank $n-r$ so that the natural map \[\rho\colon T_{\mathcal F} \oplus T_{\mathcal G} \to T_X\] is generically surjective.  For any prime divisor $D \subset X$ we define ${\rm tang}(D, \mathcal F, \mathcal G)$ to be the length of ${\rm coker}~ \rho$ at the generic point of $D$.
We define the tangency divisor of $\mathcal F$ and $\mathcal G$ to be \[{\rm tang}(\mathcal F, G) = \sum_{D \subset X}{\rm tang}(D, \mathcal F, \mathcal G) ~D.\]
\end{definition}

If $\mathcal G$ is generated by vector fields $v_1, \ldots, v_{n-r} \in T_X$ on $X$, we will denote ${\rm tang}(\mathcal F, \mathcal G)$ by 
${\rm tang}(\mathcal F, v_1,\ldots, v_{n-r})$.

\begin{remark}
Let $X$ be a normal variety of dimension $=n$, let $\mathcal F$ be a foliation on $X$ and let $v_1, \dots, v_{n-1} \in T_X$.
\begin{itemize}
\item If $\mathcal F$ is generated by a vector field $v$ and $v_1, \dots, v_{n-1}$ generate the annihilator of a $1$-form $\omega$ without codimension one zeros, then 
$$
{\rm tang}(\mathcal F, v_1, \dots, v_{n-1}) = {\rm div} (\omega(v)).
$$

    \item In the case that $\mathcal F$ is defined by a 1-form $\omega$, then ${\rm tang}(\mathcal F, v_1) = \{\omega(v_1) = 0\}$.
    \item In the case that $\mathcal F$ is defined by a vector field $w$ then 
    \[{\rm tang}(\mathcal F, v_1, \ldots, v_{n-1}) = \{w \wedge v_1 \wedge \cdots \wedge v_{n-1} = 0\}\]
    where $w \wedge v_1 \wedge \cdots \wedge v_{n-1}$
    is considered as a section of $\bigwedge^nT_X$.
\end{itemize}
\end{remark}

\begin{lemma}\label{lem:normal}
Let $X$ be a normal variety of dimension $n$ and let $\mathcal F$ be a foliation of rank $r$ on $X$. Set $N:=T_X/T_\cF$. Let $v_1,\ldots, v_{n-r}$ be vector fields on $X$ such that $v_1 \wedge \cdots \wedge v_{n-r}$ gives a non-zero element in $H^0(X,\det N)$. Then 
$$
K_\cF \sim K_X+\Gamma,
$$
where $\Gamma = {\rm tang}(\mathcal F, v_1,\ldots, v_{n-r})$ and $\cO_X(\Gamma) \sim \det N$.
\end{lemma}
\begin{proof}
Recall that $N$ is a torsion free sheaf and $\det N$ is defined as the double dual of $\wedge^{n-r} N$, in particular it is a reflexive sheaf. The same is true for $K_\cF$ and $K_X$, hence it is enough to prove the equation on an open subset $U$ of $X$ such that $\codim(X \setminus U) \ge 2$.  In particular we may assume that 
$U$ is smooth and $\cF$ is smooth. 
The statement follows now noting that the tangency locus is given by the vanishing locus of the section $\sigma=v_1 \wedge \cdots \wedge v_{n-r}$ of $\det N$
and the fact that $\mathcal O(K_{X}) \cong \mathcal O(K_{\mathcal F}) \otimes (\det N)^*$.
\end{proof}

Before turning to the proof of Theorem \ref{thm:main_intro} we state and prove some useful results.

\begin{lemma}
\label{lem_A}
Let $X$ be a quasi-projective klt variety and $B$ a reduced divisor such that $(X, B)$ is log canonical. Suppose that $T_X(-\log B)$ is locally free. Then $(X,B)$ is toroidal, in particular, if  $x \in X$ is a closed point, then 
  \begin{enumerate}
      \item there exists a Euclidean neighborhood $U$ of $x$ and a small modification $\mu\colon V \to U$ such that $V$ has finite quotient singularities; 
      \item for any point $y \in V$, there exists an open neighbourhood $W$ of $y$ and a finite quasi-\'etale morphism $q\colon W' \to W$ such that $(W', B':=q^{-1}\mu^{-1}B)$
      is an snc pair;
       and 
      \item we have natural isomorphisms 
      \[q^*\mu^*T_U(-\log B) \cong T_{W'}(-\log B')\]
      and 
      \[q^*\mu^*\Omega^{1}_U(\log B) \cong \Omega^{1}_{W'}(\log B').\]
      
  \end{enumerate}
\end{lemma}
\begin{proof}
The fact that $(X,B)$ is toroidal is \cite[Theorem 6.8]{Bergner}.  So for any $x \in X$, we can take a Euclidean neighborhood $U$ of $x$ such that $(U, D|_U)$ is isomorphic to a Euclidean open subset of a toric variety together with its torus boundary.  Up to replacing $U$ by this toric variety, we may freely assume that $U$ is a toric variety.

(1) and (2) follow from standard facts in toric geometry (and in fact the morphisms $\mu$ and $q$ may be taken to be morphisms of toric varieties).  Indeed, (1) follows by observing that any fan can be refined into a simplicial fan without adding any one dimensional rays.  (2) follows by taking $q$ to be the composition of the index one covers associated to the irreducible components of the torus boundary.

Item (3) follows by first observing $\mu^*T_{U}(-\log B) = T_V(-\log \mu^{-1}B)$.  Indeed, both of
these sheaves are locally free and are isomorphic away from ${\rm Exc}~\mu$. Since ${\rm Exc}~\mu$ is codimension $\ge 2$ we get our claimed isomorphism.  Similarly, $q$ is \'etale away from a subset of codimension $\ge 2$ and so it follows that $q^*T_W(-\log \mu^{-1}B) \cong T_{W'}(-\log B')$.  Finally, since $(T_U(-\log B))^* = \Omega^{1}(\log B)$
and $(T_{W'}(-\log B'))^* = \Omega^{1}_{W'}(\log B')$ the isomorphism 
\[q^*\mu^*T_U(-\log B) \cong T_{W'}(-\log B')\]
      immediately implies the isomorphism 
      \[q^*\mu^*\Omega^{1}_U(\log B) \cong \Omega^{1}_{W'}(\log B')\]
      as required.  
\end{proof}

\begin{proposition} \label{prop_C}
    Let $(x \in X, B)$ be a germ of a snc pair, let $\mathcal F$ be a foliation of rank 1 on $X$ and $\D$ a reduced divisor such that $(\mathcal F, \D)$ has log canonical singularities.

    Let $v_1, \ldots, v_n$ be generators of $T_X(-\log B)$. Then there is a choice of 
    $$
    {\bf b}_1=[b_{11}:b_{21}: \cdots : b_{n1}] , \ldots, {\bf b}_{n-1}=[b_{1,n-1} : \cdots : b_{n,n-1}] \in \mathbb P^{n-1}_{\mathbb C},
    $$
    such that the pair
    $$
    (X, \D+ {\rm tang}(\mathcal F,  \{ b_{11}v_1 + \ldots + b_{n1}v_{n}, \ldots, b_{1,n-1}v_1 + \ldots + b_{n,n-1}v_{n}\})
    $$
    is log canonical.
\end{proposition}

\begin{proof}
We can assume that $(X,x)$ is $(\mathbb C^n, 0)$ with coordinates $x_1, \ldots, x_n$.  
Up to reordering, we can assume that $B=\{x_1\cdots x_r=0\}$, and $T_X(-\log B)$ is generated by $v_1 := x_1 \partial_1, \ldots , v_{r}:= x_r \partial_r, v_{r+1}:= \partial_{r+1}, \ldots v_n \partial_n$.
Suppose that $\mathcal F$ is generated by a vector field $\partial$.

\medskip

Consider the logarithmic 1-form
\[
\omega = \sum_{i=1}^r \lambda_i \frac{dx_i}{x_i}+\sum_{i = r+1}^n \lambda_i dx_i,
\] 
where $\lambda_i \in \mathbb C$ are taken to be general. 

We define \[v_i =
\begin{cases}
    \lambda_ix_1\partial_1 - \lambda_1x_i\partial_i \qquad &\text{for } 2 \leq i \leq r \\
    \lambda_ix_1\partial_1 - \lambda_1\partial_i \qquad &\text{for } r+1 \leq i \leq n.
\end{cases}\]
Note that $v_2, \dots, v_n$ generate the annihilator of $\omega$
and so ${\rm tang}(\mathcal F, v_2, \dots, v_n)$
is
\[
\Gamma = \mathrm{div}(x_1\cdots x_r \,\omega(\partial)).
\]

To show that $(X,\D+\Gamma)$ is log canonical in a neighbourhood of $x$ we argue in cases based on whether $\partial$ is singular at $x$ or not.

{\bf Case 1: $\partial$ is singular at $x$.}  By Lemma \ref{lem_support}, we know that $0$ is not contained in the support of $\D$.
Hence we may simply assume that $\D=0$.
Write 
$$
\partial= \partial_0 + \partial'= \sum_{i,j}a_{ij}x_i\partial_j + \partial', 
$$
where $a_{ij} \in \mathbb C$ and $\partial'$ vanishes with order at least 2 at $0$. Let $A = (a_{ij})_{1\leq i, j\leq n}$.
%and, for $i=1,\ldots, r$, $w_i=\sum_{j} b_{ji}x_j\partial_j$, while for $i=r+1, \ldots, n-1$, $w_i=\sum_{j} b_{ji}$. 

By Lemma \ref{lem_def_normal_cone}(2), it is enough to check that
$$
(\C^n, \Gamma_0:=\mathrm{div}(x_1\cdots x_r \,  \omega(\partial_0))) = (\C^n, \mathrm{div}(x_1\cdots x_r \,  \omega(\sum_{i,j}a_{ij}x_i\partial_j))
$$
is log canonical at the origin. Note that $\Gamma_0$ is the defined by the vanishing of
$$
f(x_1,\dots,x_n) :=  (x_1 \cdots x_r)(\sum_{i=1}^r \lambda_i \frac{A_i}{x_i}+ \sum_{i=r+1}^n \lambda_i A_i)
$$
where $A_i := \sum _{j=1}^n a_{ij}x_j$. Since $\cF$ is log canonical, $A$ is not nilpotent and the conclusion follows by Proposition \ref{prop_det}.   

\medskip

{\bf Case 2: $\partial$ is non-singular at $x$.}
Write $\partial = \sum a_i \partial_i$.  By assumption for some $s \in \{1, \dots, n\}$ we have $a_s(0) \neq 0$ and so up to rescaling $\partial$ by a unit we may assume that $a_s = 1$.  
Note, moreover, that if $\Delta = \{f = 0\} \neq 0$, then $
\partial(f)$ is a unit, cf. \cite[Fact III.i.1]{MP13}, so without loss of generality we may assume that $f = x_s+ f_0(x_1, \dots, x_n)$.

A direct calculation  shows that 
\[{\rm tang}(\mathcal F,  \{v_1, \ldots,  \widehat{v_s},\ldots,  v_n\} = \{\prod_{1 \leq i \leq r,  i \neq s} x_i + g_0 = 0\}\]
where ${\rm ord}_0g_0 > {\rm ord}_0\prod_{1 \leq i \leq r,  i \neq s} x_i$.
Another application of 
Lemma \ref{lem_def_normal_cone}(2) allows us to conclude that 
$$
(\mathbb C^n, \{x_s+f_0 = 0\}+\{\prod_{1 \leq i \leq r,  i \neq s} x_i + g_0 = 0\})
$$
is log canonical as required.

\end{proof}

\section{Proof of Theorem \ref{thm:main_intro}}

We are now ready to prove our main Theorem.

\begin{proof}[Proof of Theorem \ref{thm:main_intro}] 
Let $w_1, \dots, w_m$ be a basis of $H^0(X, T_X(-\log B)\otimes D)$.
For $i = 1, \dots, n-1$ define $v_i = \sum_{j = 1}^m a_{ij}w_j$
where $a_{ij} \in \mathbb C$.
We will show for a general choice of $a_{ij}$ that 
the pair $(X, \Delta+{\rm tang}(\mathcal F, v_1, \dots v_n))$ is log canonical.

Let $x \in X$ be any point. 
    By Lemma \ref{lem_A} there exists a Euclidean neighbourhood $U$ of $x$ together with a small modification $\mu\colon V \to U$ such that $V$ has only cyclic quotient singularities.  Let $y \in V$ be any point an let $W' \to W$ be the cyclic quotient where $(W', B':=q^{-1}\mu^{-1}B)$ is
    an snc pair.  
    
We can pullback $w_1, \dots, w_m$ to $V$ to give meromorphic vector fields on $\tilde{w}_1, \dots, \tilde{w}_m$ on $V$.  Since ${\rm Exc} ~ \mu$ is codimension $\ge 2$ we see that in fact $\tilde{w}_1, \dots, \tilde{w}_m$ are holomorphic vector fields on $V$.  Similarly, we can pullback $w_1, \dots, w_m$ to holomorphic vector fields $w'_1, \dots, w'_m$ on $W'$.
    Let us also write \[K_{V}+\Delta_V = \mu^*(K_U+\Delta|_U)\] and \[K_{W'}+\Delta_{W'} = q^*(K_W+\Delta_V|_W).\]

    Since $\mu$ (resp. $q$) is an isomorphism (resp. is \'etale) away from a subset of codimension $\ge 2$ we have the following equalities of divisors: 
    \begin{align}
    \label{eqn_1}
    \mu^*{\rm tang}(\mathcal F, v_1, \dots, v_{n-1}) = {\rm tang}(\mu^{-1}\mathcal F, \tilde{v}_1, \dots, \tilde{v}_{n-1})
    \end{align}
    and 
    \begin{align}
    \label{eqn_2}
    q^*{\rm tang}(\mu^{-1}\mathcal F, \tilde{v}_1, \dots, \tilde{v}_{n-1}) = {\rm tang}(q^{-1}\mu^{-1}\mathcal F, v'_1, \dots, v'_{n-1})
    \end{align}
    where $\tilde{v}_i$ (resp. $v'_i$) is the pullback of $v_i$ to $V$ (resp. $W'$).
Note that ${\rm tang}(\mathcal F, v_1, \dots, v_{n-1}) \in | \det N_{\mathcal F}|$,
${\rm tang}(\mu^{-1}\mathcal F, \tilde{v}_1, \dots, \tilde{v}_{n-1}) \in |\det N_{\mu^{-1}\mathcal F}|$
are both $\mathbb Q$-Cartier divisors, see Lemma \ref{lem:normal}.

    Again, by Lemma \ref{lem_A} we see that $w'_1, \dots, w'_m$ generate $T_{W'}(-\log B')$.
    By Proposition \ref{prop_C}, if $a_{ij}$ are general, it follows that $(W', \Delta_{W'}+{\rm tang}(q^{-1}\mu^{-1}\mathcal F, v'_1, \dots, v'_{n-1}))$
    is log canonical. By \cite[Proposition 5.20]{KM98} and (\ref{eqn_2}) we see that 
    $(W, \Delta+{\rm tang}(\mu^{-1}\mathcal F, \tilde{v}_1, \dots, \tilde{v}_{n-1}))$
    is log canonical.

    By the Bertini-Koll\'ar theorem, \cite[Theorem 4.8]{Kollar95}, for a general choice of $a_{ij}$
    we have
$(V, \Delta_V+{\rm tang}(\mu^{-1}\mathcal F, \tilde{v}_1, \dots, \tilde{v}_{n-1}))$ is log canonical (we remark that although \cite[Theorem 4.8]{Kollar95} is stated for schemes, the proof works equally well for complex analytic varieties). 
By (\ref{eqn_1}) 
\begin{align*}
K_V+\Delta_V+{\rm tang}(\mu^{-1}\mathcal F, \tilde{v}_1, \dots, \tilde{v}_{n-1}) = \\ \mu^*(K_U+(\Delta+{\rm tang}(\mathcal F, v_1, \dots, v_{n-1}))|_U)
\end{align*}
and we deduce that $(U, (\Delta+{\rm tang}(\mathcal F, v_1, \dots, v_{n-1}))_U$ is log canonical for a general choice of $a_{ij}$.
Another application of the Bertini-Koll\'ar theorem, \cite[Theorem 4.8]{Kollar95}
allows us to conclude that for a general choice of $a_{ij}$, 
$(X, \Delta+{\rm tang}(\mathcal F, v_1, \dots, v_{n-1}))$ is log canonical.

\medskip

We conclude by observing that the linear equivalence $K_X+\D + \Gamma \sim K_{\mathcal F}+ \D + D$ follows from Lemma \ref{lem:normal}. 
\end{proof}

\section{Proof of corollaries}

\begin{proof}[Proof of Corollary \ref{cor:DCC}]
We may apply Theorem \ref{thm:main_intro} with  $D = 0$ to find a $\mathbb Z$-divisor $\Gamma$ so that $(X, \D+ \Gamma)$ is log canonical and $K_X+ \D + \Gamma \sim K_{\mathcal F}+\D$.
We may then conclude by \cite[Theorem 1.3]{HMX12}. 
\end{proof}

The following example shows that the log canonicity assumption is essential. 

\begin{example}
\label{ex_toric_not_lc}
    Let $n \in \mathbb Z_{>0}$ and consider the foliation $\mathcal F$ on $\mathbb P(1, 1, n)$
    defined by the homogeneous 1-form $\omega = X^2ZdX+Y^2ZdY-(X^3+Y^3)dZ$.
    Since the weight of $\omega$ is $3+n$ we see that $\mathcal O(K_{\mathcal F}) \cong \mathcal O(1)$, hence $K_{\mathcal F}^2 =\frac{1}{n}$.

However, in a neighbourhood of the unique singular point on $\mathbb P(1, 1, n)$, the foliation is a quotient of the foliation defined (in appropriate coordinates) by the vector field $y^2\partial_x-x^2\partial_y$.
    This vector field has linear part $=0$, and is therefore not log canonical.
\end{example}

\begin{proof}[Proof of Corollary \ref{cor:boundedness}] 
For any $(X,\D,\mathcal F)$ as in the statement, we may apply Theorem \ref{thm:main_intro} with  $D = 0$ to find a $\mathbb Z$-divisor $\Gamma$ so that $(X, \D+ \Gamma)$ is log canonical and $K_X+ \D + \Gamma \sim K_{\mathcal F}+\D$.
The pairs $(X, \D+ \Gamma)$ are log canonical models of general type with fixed volume and $\mathrm{coeff}(\D+\Gamma) \in \{1\}$; hence they are bounded, see  \cite{HMX14}. 

The conclusion follows now applying \cite[Proposition 3.36]{C7}.
\end{proof}

\begin{proof}[Proof of Corollary \ref{cor:MMP}]
By Theorem \ref{thm:main_intro} there is a $\mathbb Z$-divisor $\Gamma$ so that $(X, \D+ \Gamma)$ is log canonical and $K_X+ \D + \Gamma \sim K_{\mathcal F}+\D$.
By  \cite{HX11} we can run a $K_X+\D+\Gamma$ MMP which will also be a $K_{\mathcal F}+\D$-MMP. Since $\cF$ and $\D$ are $G$-invariant, so is $\Gamma$. Given that $G$ is a connected group, any MMP is $G$-equivariant (see \cite[\S 4.4]{BF21}). At each step, there is an action of $G$ on $(X_i, \D_i, \cF_i)$, which leaves $\D_i$ and $\cF_i$ invariant and for which $X_i \setminus \D_i$ is an open orbit. Hence the variety is still log homogeneous.
\end{proof}

\begin{remark}
\label{rem_scaling}
    Since toroidal varieties (and in particular log homogeneous varieties) are potentially klt the above proof also shows that we can run a $K_{\mathcal F}+\Delta$-MMP with scaling of any ample divisor and that such an MMP terminates if either $\Delta$ is big or $K_{\mathcal F}+\Delta$ is not pseudo-effective, cf. \cite{BCHM}.
\end{remark}

\begin{proof}[Proof of Corollary \ref{cor:adjoint_minimal}]
By Theorem \ref{thm:main_intro} there is a $\mathbb Z$-divisor $\Gamma$ so that $(X, \Gamma)$ is log canonical and $K_X + \Gamma \sim K_{\mathcal F}$. Then
$$
tK_\cF + (1-t)K_X \sim K_X+t\Gamma
$$
is nef and big with volume $v$. Since $t <1$, the pair $(X,t\Gamma)$ is klt with coefficients in the DCC set $\{t\}$.  By \cite{MST}, such pairs are bounded and we conclude by \cite[Proposition 3.36]{C7}.
\end{proof}

\begin{proof}[Proof of Corollary \ref{cor:adjoint_Fano}]
By Theorem \ref{thm:main_intro} there is a $\mathbb Z$-divisor $\Gamma$ so that $(X, \Gamma)$ is log canonical and $K_X + \Gamma \sim K_{\mathcal F}$. Then
$$
-(tK_\cF + (1-t)K_X) \sim -(K_X+t\Gamma)
$$
is ample.  Since $ t \le (1-\varepsilon)$, the pair $(X, t \Gamma)$ is $\varepsilon^2$-lc. In fact, for any divisor $E$ over $X$,
\[
a(E,X,t\Gamma)=a(E,X,0)-t\mult_E(\Gamma).
\]
Since $(X,\Gamma)$ is lc, $\mult_E(\Gamma)\le a(E,X,0)+1$, hence
\[
a(E,X,t\Gamma)\ge (1-t)a(E,X,0)-t.
\]
As $X$ is $\varepsilon$-lc, $a(E,X,0)\ge -1+\varepsilon$, so
\[
a(E,X,t\Gamma)\ge (-1+\epsilon)(1-t)-t = -1+(1-t)\varepsilon \ge -1+\varepsilon^2,
\]
where the last inequality holds because $\varepsilon \le 1-t$. Thus $(X,t\Gamma)$ is $\varepsilon^2$-lc.

Applying \cite{Bir21}, we see that the pairs $(X,t\Gamma)$ form a bounded family and then we obtain the boundedness of foliations by \cite[Proposition 3.36]{C7}. In this last step, it is important that $t >0$.
\end{proof}

\begin{proof}[Proof of Corollary \ref{cor:Fano}]
By Theorem \ref{thm:main_intro} there is a $\mathbb Z$-divisor $\Gamma$ so that $(X, \Gamma)$ is log canonical and $K_X + \Gamma \sim K_{\mathcal F}$. Then
$$
-(tK_\cF + (1-t)K_X) \sim -(K_X+t\Gamma)
$$
is ample. For $\varepsilon >0$ small enough  $(X,(t-\varepsilon)\Gamma)$ is klt and $-K_X-t\Gamma - \varepsilon \Gamma$ is still ample.
\end{proof}

\bibliographystyle{alpha}
\bibliography{bib}

\end{document}